\documentclass{article}
\usepackage{amsmath,amssymb,amsthm}
\usepackage{parskip}
\usepackage{authblk}
\usepackage{tikz}
\usepackage[numbers]{natbib}
\usepackage{geometry}
\usepackage{multicol}
\usepackage{enumerate}
\usepackage[bookmarks=false,colorlinks=true,citecolor=blue]{hyperref}
\usepackage[mathscr]{euscript}
\usepackage{mathtools}

\DeclarePairedDelimiter\ceil{\lceil}{\rceil}
\DeclarePairedDelimiter\floor{\lfloor}{\rfloor}

\newtheorem{theorem}{Theorem}[section]
\newtheorem{conjecture}[theorem]{Conjecture}

\newtheorem{definition}[theorem]{Definition}

\newtheorem{observation}[theorem]{Observation}
\newtheorem{case}{Case}{\bfseries}{\itshape}
\newtheorem{subcase}{Subcase}
\numberwithin{subcase}{case}

\title{\textbf{An improved upper bound for the domination number of a graph}}
\author[1]{\textbf{Subramanian~Arumugam}}
\author[2]{\textbf{Suresh~Manjanath~Hegde}}
\author[2]{\textbf{Shashanka~Kulamarva}\thanks{Corresponding Author, Current Affiliation: Indian Institute of Science, Bangalore-560012, India\\ Email: shashankak@iisc.ac.in, shashank.klm@gmail.com, ORCID ID: 0009-0002-2982-6044}}
\affil[1]{Ramco Institute of Technology, Rajapalayam-626117, India}
\affil[2]{National Institute of Technology Karnataka, Surathkal-575025, India\protect\\ Email: \texttt{s.arumugam.klu@gmail.com, smhegde@nitk.edu.in, skulamarva.187ma007@nitk.edu.in}}
\date{}

\begin{document}
	\maketitle
	\begin{abstract}
		\noindent Let $G$ be a graph of order $n$. A classical upper bound for the domination number of a graph $G$ having no isolated vertices is $\floor{\frac{n}{2}}$. However, for several families of graphs, we have $\gamma(G) \le \floor{\sqrt{n}}$ which gives a substantially improved upper bound. In this paper, we give a condition necessary for a graph $G$ to have $\gamma(G) \le \floor{\sqrt{n}}$, and some conditions sufficient for a graph $G$ to have $\gamma(G) \le \floor{\sqrt{n}}$. We also present a characterization of all connected graphs $G$ of order $n$ with $\gamma(G) = \floor{\sqrt{n}}$. Further, we prove that for a graph $G$ not satisfying $rad(G)=diam(G)=rad(\overline{G})=diam(\overline{G})=2$, deciding whether $\gamma(G) \le \floor{\sqrt{n}}$ or $\gamma(\overline{G}) \le \floor{\sqrt{n}}$ can be done in polynomial time. We conjecture that this decision problem can be solved in polynomial time for any graph $G$.\\
		
		
	\end{abstract}
	
	\section{Introduction}
	By a graph $G=(V,E)$, we mean a finite undirected graph with neither loops nor multiple edges. The order $|V|$ and size $|E|$ of $G$ are represented by $n$ and $m$ respectively. For basic graph theoretic terminologies, we refer to the book \cite{Chartrand2015GD}.
	
	Let $G=(V,E)$ be a graph. A subset $D$ of $V$ is said to be a dominating set of $G$ if every vertex $v \in V \setminus D$ is adjacent to a vertex in $D$. The domination number of $G$ denoted by $\gamma(G)$, is the minimum cardinality of a dominating set of $G$. A dominating set $D$ of $G$ with $|D|=\gamma$ is called a $\gamma$-set of $G$.
	
	A dominating set $D$ of $G$ is said to be a total dominating set if the induced subgraph $G[D]$ has no isolated vertices. Further, a dominating set $D$ of $G$ is said to be a connected dominating set if the induced subgraph $G[D]$ is a connected subgraph of $G$. The total domination number $\gamma_t$ and the connected domination number $\gamma_c$ of $G$ are defined as follows.
	\begin{align*}
		\gamma_t(G) & = \text{min }\{|D| : \text{$D$ is a total dominating set of $G$}\}\\
		\gamma_c(G) & = \text{min }\{|D| : \text{$D$ is a connected dominating set of $G$}\}
	\end{align*}
	A total dominating set $D$ of $G$ with $|D|=\gamma_t$ is called a $\gamma_t$-set of $G$. Further, a connected dominating set $D$ of $G$ with $|D|=\gamma_c$ is called a $\gamma_c$-set of $G$. We observe that $\gamma_t$ is defined only for graphs without isolated vertices and $\gamma_c$ is defined only for connected graphs. For domination related concepts, we refer to the book \cite{Haynes1998Dom}.
	
	The concept of domination has significant applications in Network Science. \citet{Brewster2019BroadDom} have investigated the concept of broadcast domination in the context of communication systems. For specific applications of dominating sets in social networks and biological networks, we refer to \cite{Kelleher1988DomSocNtwk} and \cite{Nacher2016DomBioNtwk}. \citet{Cheng2003ApprxSchmDomAdhoc} demonstrated that the problem of constructing a connected dominating set in a unit disc graph is the same as constructing a virtual backbone in Ad Hoc wireless networks. In all such applications, the number of vertices in the network is large and the crucial problem is to find a dominating set of small order. In this paper, we prove that for most of the graphs, there exists a dominating set of order at most $\floor{\sqrt{n}}$, which is quite significant in the context of the above mentioned applications.
	
	Trivially, $1 \le \gamma \le n$ and $\gamma = n$ if and only if $G = \overline{K_n}$. For graphs without isolated vertices, a classical result of \citet{Ore1962ThrGrphs} gives an improved upper bound for $\gamma$.
	
	\begin{theorem}\cite{Ore1962ThrGrphs}
		If a graph $G$ has no isolated vertices, then $\gamma(G) \le \frac{n}{2}$.
	\end{theorem}
	
	For a graph theoretic parameter, Nordhaus-Gaddum type result gives an upper bound and a lower bound for the sum and product of the values of the parameter for a graph $G$ and its complement $\overline{G}$. An excellent survey of Nordhaus-Gaddum type results for various parameters is given in \cite{Aouchiche2013NordGaddSurvey}. In the classical paper \cite{Nordhaus1956ComGrphs}, Nordhaus and Gaddum established the first such result for the chromatic number. \citet{Jaeger1972Dom} obtained similar results for the domination number. We state only the multiplicative version of the theorem.
	
	\begin{theorem}\cite{Jaeger1972Dom}\label{thm:JaegerDomProduct}
		Let $G$ be a graph of order $n$. Then $\gamma(G)\gamma(\overline{G}) \le n$.
	\end{theorem}
	
	Theorem~\ref{thm:JaegerDomProduct} leads to the following observation giving a much improved upper bound for the domination number.
	
	\begin{observation}\label{obs:GOrGCompHasLesserDom}
		If $G$ is any graph of order $n$, then either $\gamma(G) \le \floor{\sqrt{n}}$ or $\gamma(\overline{G}) \le \floor{\sqrt{n}}$.
	\end{observation}
	
	In this paper, we investigate the structure of graphs satisfying $\gamma(G) \le \floor{\sqrt{n}}$ and $\gamma(\overline{G}) \le \floor{\sqrt{n}}$. The following definitions and theorems are necessary for our discussions.
	
	\begin{theorem}\cite{Berge1982ThrGrphsAndAppl, Walikar1979Dom}\label{thm:DomBoundDelta}
		For any graph $G$ with maximum degree $\Delta$, we have
		\begin{equation*}
			\ceil*{\frac{n}{1+\Delta}} \le \gamma(G) \le n-\Delta
		\end{equation*}
	\end{theorem}
	
	For any vertex $v \in V$, the set $N(v) = \{u \in V : uv \in E\}$ is called the open neighborhood of $v$. Also, $N[v] = N(v) \cup \{v\}$ is called the closed neighborhood of $v$.
	
	\begin{definition}
		Let $G=(V,E)$ be a graph of order $n$. Let $S \subseteq V$ and let $v \in S$. A vertex $w$ is called a private neighbor of $v$ if $N[w] \cap S = \{v\}$. If $w \in V \setminus S$, then $w$ is called an external private neighbor of $v$.
	\end{definition}
	
	Let $pn[v,S]$ and $epn[v,S]$ denote the set of all private neighbors of $v$ and the set of all external private neighbors of $v$, respectively. One can see that $v$ is a private neighbor of itself if and only if $v$ is an isolated vertex in the induced subgraph $G[S]$. The following theorem mentioned in \cite{Haynes1998Dom} is a consequence of a result proved by \citet{Ore1962ThrGrphs}.
	
	\begin{theorem}\cite{Ore1962ThrGrphs}\label{thm:MinDomSetCondition}
		A dominating set $D$ of the graph $G$ is a minimal dominating set if and only if $pn[v,D] \neq \emptyset$ for all $v \in D$. If $v$ is not an isolated vertex in the induced subgraph $G[D]$, then $epn[v,D] \neq \emptyset$.
	\end{theorem}
	
	The following theorem is proved by \citet{Brigham1988Dom}.
	
	\begin{theorem}\cite{Brigham1988Dom}\label{thm:DiaMoreThan2InComImpliesDomLessThan3}
		If $\gamma(\overline{G}) \ge 3$, then $diam(G) \le 2$.
	\end{theorem}
	
	\section{Main Results}
	We present several structural results on graphs satisfying $\gamma(G) \le \floor{\sqrt{n}}$. The following theorem gives a necessary condition for a graph $G$ to have $\gamma(G) \le \floor{\sqrt{n}}$.
	
	\begin{theorem}
		Let $G$ be a graph of order $n \ge 2$ and let $\gamma(G) \le \floor{\sqrt{n}}$. Then $\Delta \ge \ceil{\sqrt{n}} - 1$ and the bound is sharp.
	\end{theorem}
	
	\begin{proof}
		From Theorem~\ref{thm:DomBoundDelta}, we have $\gamma(G) \ge \ceil*{\frac{n}{1+\Delta}}$. Also $\gamma(G) \le \floor{\sqrt{n}}$ and hence $\ceil*{\frac{n}{1+\Delta}} \le \floor{\sqrt{n}}$.
		
		Therefore, $\frac{n}{1+\Delta} \le \floor{\sqrt{n}}$, which implies $\Delta \ge \ceil{\sqrt{n}} - 1$.
		
		For the cycle $C_5$, $\gamma = 2 = \floor{\sqrt{n}}$ and $\Delta = 2 = \ceil{\sqrt{n}} - 1$. Hence the bound is sharp.
	\end{proof}
	
	We now proceed to characterize connected graphs of order $n$ with $\gamma(G) = \floor{\sqrt{n}}$. For this purpose, we introduce a family of graphs.
	
	\begin{definition}\label{dfn:SpecialGraphClass}
		Let $H_1$ be a graph of order $k$ with the vertex set $V(H_1) = \{v_1, v_2, \dots, v_k\}$. Let $H_2$ be another graph satisfying:
		\begin{equation*}
			k(k-1) \le |V(H_2)| \le k^2+k
		\end{equation*}
		Let $\{V_1, V_2, \dots, V_k, V_{k+1}\}$ be a partition of $V(H_2)$ such that whenever $v_i$ is not an isolated vertex in $H_1$, $V_i \neq \emptyset$ for any $i$ with $1 \le i \le k$ (The set $V_{k+1}$ may be empty and $V_i$ may be empty for $1 \le i \le k$ if $v_i$ is an isolated vertex in $H_1$).
		
		Let $G$ be the graph obtained from $H_1$ and $H_2$ as follows.
		\begin{enumerate}
			\item For $1 \le i,j \le k$ with $i \neq j$, the vertex $v_i$ is adjacent to all the vertices in $V_i$ and not adjacent to any vertex in $V_j$.
			\item Each vertex of $V_{k+1}$ is adjacent to at least two vertices of $H_1$.
		\end{enumerate}
		
		For any subset $I \subseteq \{1, 2, \dots ,k\}$, we define two induced subgraphs as follows:
		\begin{align*}
			& H_1^I = H_1[\{v_i : i \in I\}]\\
			& H_2^I = H_2[(\bigcup_{i \in I} V_i) \cup \{v \in V_{k+1} : N(v) \cap \{v_i : i \notin I\} = \emptyset\}]
		\end{align*}
		For any two subsets $S \subseteq V(H_2)$ and $I \subseteq \{1, 2, \dots ,k\}$, we define two sets as follows:
		\begin{equation*}
			N_{H_1^I}(\overline{I}) = \bigcup_{i \notin I} N_{H_1^I}(v_i) \quad \text{and} \quad N_{H_1}(S) = \bigcup_{v \in S}^{} N_{H_1}(v)
		\end{equation*}
		We further assume that the graph $G$ satisfies the following condition:
		\begin{enumerate}[$(C).$]
			\item For any subset $I \subseteq \{1, 2, \dots ,k\}$ with $|I| = s$, if $S$ is a dominating set of $H_2^I$ with $|S| \le s-1$, then $|N_{H_1}(S) \cup N_{H_1^I}(\overline{I})| \le s-1$.
		\end{enumerate}
	\end{definition}
	
	Let $\mathscr{F}$ denote the family of graphs $G$ constructed as above. An example of a graph that belongs to the family $\mathscr{F}$ is shown in Figure~\ref{fig:Example}.
	
	\begin{figure}[h]
		\centering
		\bigskip
		\begin{tikzpicture}
			\begin{scope}[every node/.style={circle,draw,inner sep=0pt, minimum size=5ex}]
				\node (v1) at (0,-2) {$v_1$};
				\node (v11) at (-0.6,2) {$v_1^1$};
				\node (v12) at (0.6,2) {$v_1^2$};
				\node (v2) at (2.5,-2) {$v_2$};
				\node (v21) at (1.9,2) {$v_2^1$};
				\node (v22) at (3.1,2) {$v_2^2$};
				\node (vk) at (6.5,-2) {$v_k$};
				\node (vk1) at (5.9,2) {$v_k^1$};
				\node (vk2) at (7.1,2) {$v_k^2$};
				\node (vk11) at (8.5,2) {$v_{k+1}^1$};
				\node (vk12) at (9.5,2) {$v_{k+1}^2$};
				\node (vk13) at (10.5,2) {$v_{k+1}^3$};
			\end{scope}
			\node (x1) at (3.5,1) {};
			\node (x2) at (5.5,1) {};
			\node (x3) at (4,-2) {};
			\node (x4) at (5,-2) {};
			\node[label=above:$V_1$] at (0,2.55) {};
			\draw (0,2) ellipse (1cm and 0.75cm);
			\draw (v1) to (v11);
			\draw (v1) to (v12);
			\node[label=above:$V_2$] at (2.5,2.55) {};
			\draw (2.5,2) ellipse (1cm and 0.75cm);
			\draw (v2) to (v21);
			\draw (v2) to (v22);
			\draw [dashed] (x1) to (x2);
			\draw (vk) to (vk1);
			\draw (vk) to (vk2);
			\node[label=above:$V_k$] at (6.5,2.55) {};
			\draw (6.5,2) ellipse (1cm and 0.75cm);
			\node[label=above:$V_{k+1}$] at (9.5,2.55) {};
			\draw (9.5,2) ellipse (1.4cm and 0.75cm);
			\draw (v1) to (v2);
			\draw (v2) to (x3);
			\draw (vk) to (x4);
			\draw (vk11) to (v1);
			\draw (vk11) to (v2);
			\draw (vk12) to (v1);
			\draw (vk12) to (v2);
			\draw (vk13) to (v2);
			\draw (vk13) to (vk);
			\node[label=above:$H_1$] at (-1,-2.4) {};
			\draw (3.25,-2) ellipse (4cm and 0.75cm);
			\node[label=above:$H_2$] at (-1.5,1.6) {};
			\draw (5,2) ellipse (6.25cm and 2cm);
		\end{tikzpicture}
		\caption{An example of a graph in the family $\mathscr{F}$}
		\label{fig:Example}
		\bigskip
	\end{figure}
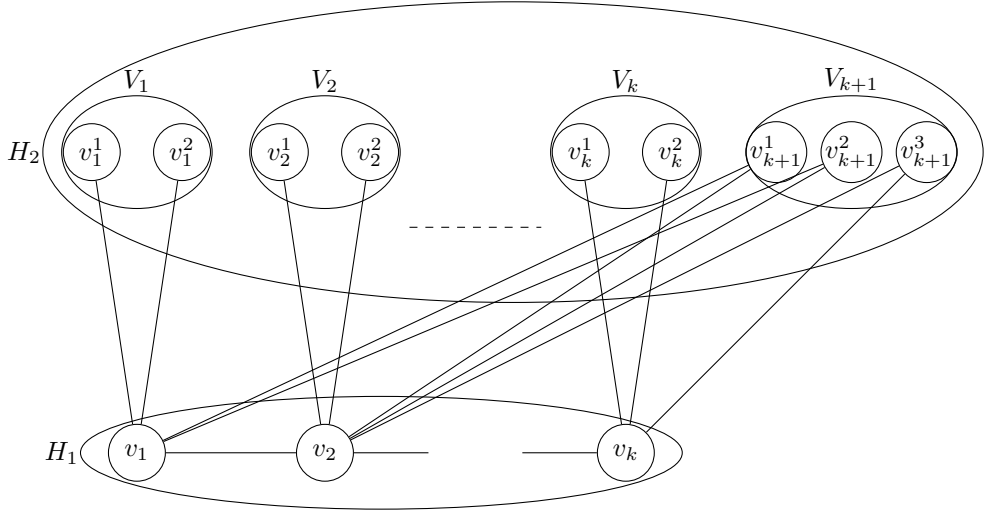
	
	We now proceed to prove that $\mathscr{F}$ is precisely the family of all connected graphs of order $n$ with $\gamma(G) = \floor{\sqrt{n}}$.
	
	\begin{theorem}
		Let $G$ be a connected graph of order $n$ with $n \ge 4$. Then $\gamma(G) = \floor{\sqrt{n}}$ if and only if $G \in \mathscr{F}$.
	\end{theorem}
	
	\begin{proof}
		Let $G \in \mathscr{F}$. Then, $G$ is a graph obtained from some $H_1$ and $H_2$ as in Definition~\ref{dfn:SpecialGraphClass}. Therefore, we have $n = |V(G)| = k + |V(H_2)|$ and $k(k-1) \le |V(H_2)| \le k^2+k$. Hence, $k^2 \le n < (k+1)^2$, and so $k = \floor{\sqrt{n}}$. Clearly, $V(H_1)$ is a dominating set of $G$ and hence, $\gamma(G) \le \floor{\sqrt{n}}$. Now, let $D$ be any $\gamma$-set of $G$. Suppose $|D| \le k-1$. We consider two cases.
		
		\begin{case}\label{case:DLiesInH1}
			$D \cap V(H_2) = \emptyset$.
		\end{case}
		
		In this case, $D \subseteq V(H_1)$. Since $|D| \le k-1$, there exists a vertex $v_j \in V(H_1) \setminus D$. If $V_j \neq \emptyset$, then $N(w) \cap D = \emptyset$ for all $w \in V_j$. Otherwise, $v_j$ is an isolated vertex in $H_1$ and $N(v_j) \cap D = \emptyset$. Hence, in any case, $D$ is not a dominating set of $G$, which is a contradiction.
		
		\begin{case}
			$D \cap V(H_2) \neq \emptyset$.
		\end{case}
		
		Let $I = \{i : 1 \le i \le k, v_i \notin D\}$. Since $|D| \le k-1$ and $D \cap V(H_2) \neq \emptyset$, it follows that $I \neq \emptyset$. Let $|I| = s$. We claim that $S = D \setminus \{v_i : i \notin I\}$ is a dominating set of $H_2^I$. Let $v \in V(H_2^I) \setminus S$. Let $u$ be a vertex in $D$ which dominates $v$. From the definition of $H_2^I$, we have two subcases.
		
		\begin{subcase}
			$v \in \bigcup\limits_{i \in I} V_i$
		\end{subcase}
		
		Let $v \in V_j$ for some $j \in I$. Then $v_j \notin D$ and hence, $u \in S$.
		
		\begin{subcase}
			$v \in V_{k+1}$ and $N(v) \cap V_i = \emptyset$ for all $i \notin I$
		\end{subcase}
		
		Since $u \in N(v)$, it follows that $u \in V_j$ for some $j \in I$. Hence, $u \in S$.
		
		In any case, we have that $u \in S$. Thus $S$ dominates all the vertices of $V(H_2^I)$ and $|S| = |D| - (k-|I|)$. Since $|D| \le k-1$ and $|I| = s$, we have $|S| \le s-1$. Hence, it follows from condition $C$ of Definition~\ref{dfn:SpecialGraphClass} that
		\begin{equation}\label{eqn:ConditionCImplication}
			|N_{H_1}(S) \cup N_{H_1^I}(\overline{I})| \le s-1
		\end{equation}
		Now, let $v_i$ be a vertex with $i \in I$. Let $u$ be a vertex in $D$ which dominates $v_i$. If $u \in S$, then $v_i \in N_{H_1}(S)$. Otherwise, $u = v_j$ for some $j \notin I$ and hence, $v_i \in N_{H_1^I}(\overline{I})$. Thus $v_i \in N_{H_1}(S) \cup N_{H_1^I}(\overline{I})$. Therefore, we have $I \subseteq N_{H_1}(S) \cup N_{H_1^I}(\overline{I})$. Hence, $|N_{H_1}(S) \cup N_{H_1^I}(\overline{I})| \ge s$, which contradicts (\ref{eqn:ConditionCImplication}).
		
		Hence, it follows that $|D| \ge k$ and $\gamma(G) \ge k$. Thus $\gamma(G) = k = \floor{\sqrt{n}}$.
		
		Conversely, let $G$ be a connected graph of order $n$ with $\gamma(G) = k = \floor{\sqrt{n}}$. Since $n \ge 4$, we have $k \ge 2$. Let $S = \{v_1, v_2, \dots, v_k\}$ be a $\gamma$-set of $G$. Let $H_1 = G[S]$ and $H_2 = G[V \setminus S]$. Since $S$ is a $\gamma$-set of $G$, if any $v_i \in S$ is not an isolated vertex in $G[S]$, then by Theorem~\ref{thm:MinDomSetCondition}, we have $epn[v_i, S] \neq \emptyset$. Let $V_i = epn[v_i, S]$, for $1 \le i \le k$. Let $V_{k+1} = \{v \in V(H_2) : |N(v) \cap S| \ge 2\}$.
		
		Clearly, $\{V_1, V_2, \dots, V_k, V_{k+1}\}$ is a partition of $V(H_2)$ and for any $i$ with $1 \le i \le k$, whenever $v_i$ is not an isolated vertex in $H_1$, we have $V_i \neq \emptyset$, the vertex $v_i$ is adjacent to each vertex in $V_i$ and is not adjacent to any vertex in $V_j$ for all $j \neq i$, and $1 \le i,j \le k$. Also, each vertex of $V_{k+1}$ is adjacent to at least two vertices of $H_1$. Since $\floor{\sqrt{n}} = k$, we have $k^2 \le n < (k+1)^2$, and hence, it follows that $k(k-1) \le |V(H_2)| \le k^2+k$.
		
		Now, suppose condition $C$ of Definition~\ref{dfn:SpecialGraphClass} does not hold for $G$. Then there exists a subset $I \subseteq \{1, 2, \dots ,k\}$ with $|I| = s$, and a dominating set $S$ of $H_2^I$ with $|S| \le s-1$ and $|N_{H_1}(S) \cup N_{H_1^I}(\overline{I})| \ge s$. We claim that $S' = S \cup \{v_i: i \notin I\}$ is a dominating set of $G$.
		
		Let $v \in V(G) \setminus S'$. If $v \in V(H_2^I)$, then $v$ is adjacent to some vertex in $S$. Otherwise, if $v \in V_j$ for some $j$ with $j \notin I$, then $v$ is adjacent to the vertex $v_j$ in $S'$. Now, let $v \in V_{k+1}$. Since $v \notin V(H_2^I)$, $v$ is adjacent to a vertex $v_j$ for some $j$ with $j \notin I$.
		
		Finally, let $v \in V(H_1)$. Then $v = v_i$ for some $i \in I$. Since $|N_{H_1}(S) \cup N_{H_1^I}(\overline{I})| \ge s = |I|$, it follows that $v_i$ is adjacent to some vertex in $S'$. Thus $S'$ is a dominating set of $G$. Also, we have
		\begin{equation*}
			|S'| = |S| + |\{v_i: i \notin I\}| \le (s-1) + (k-|I|) = (s-1)+(k-s) = k-1
		\end{equation*}
		Thus $|S'| \le k-1$, which is a contradiction to the fact that $\gamma(G) = k$. Hence, condition $C$ of Definition~\ref{dfn:SpecialGraphClass} holds for $G$ and $G \in \mathscr{F}$.
	\end{proof}
	
	We now proceed to find sufficient conditions for a graph $G$ to have $\gamma(G) \le \floor{\sqrt{n}}$.
	
	\begin{theorem}\label{thm:SuffiCondnDom}
		Let $G$ be a graph of order $n \ge 2$. If any one of the following conditions holds, then $\gamma(G) \le \floor{\sqrt{n}}$.
		
		\begin{enumerate}
			\item $\overline{G}$ is a disconnected graph.
			\item $diam(\overline{G}) \ge 3$.
		\end{enumerate}
	\end{theorem}
	
	\begin{proof}
		$(i)$. Since $\overline{G}$ is a disconnected graph, it follows that $G = K_2$ if $n=2$, and $G = P_3$ or $K_3$ if $n=3$. In both these cases, $\gamma(G) = 1 \le \floor{\sqrt{n}}$. Now, let $n \ge 4$. Let $C_1$ and $C_2$ be two components of $\overline{G}$. Let $x \in V(C_1)$ and $y \in V(C_2)$. It can be easily verified that $\{x,y\}$ is a dominating set of $G$ and hence, $\gamma(G) \le 2 \le \floor{\sqrt{n}}$.
		
		$(ii)$. Let $diam(\overline{G}) \ge 3$. Then we have $n \ge 4$. Also if $x,y \in V$ and $d_{\overline{G}}(x,y) \ge 3$, then $\{x,y\}$ is a dominating set of $G$. Alternatively, since $diam(\overline{G}) \ge 3$, by Theorem~\ref{thm:DiaMoreThan2InComImpliesDomLessThan3}, we have $\gamma(G) \le 2$. Thus $\gamma(G) \le 2 \le \floor{\sqrt{n}}$.
	\end{proof}
	
	Now, we proceed towards the problem of deciding for a given graph $G$ whether $\gamma(G) \le \floor{\sqrt{n}}$ or $\gamma(\overline{G}) \le \floor{\sqrt{n}}$. For this purpose, we need the following definitions.
	
	\begin{definition}
		A graph $G$ is said to be a $t$-complementary self-centered graph, denoted by $SCC(t)$ if both $G$ and $\overline{G}$ are self-centered graphs satisfying the conditions:
		\begin{equation*}
			rad(G)=diam(G)=rad(\overline{G})=diam(\overline{G})=t
		\end{equation*}
	\end{definition}
	
	\begin{definition}
		Let $G$ be a graph of order $n$. Then $G$ is said to be in domination type-I if $\gamma(G) \le \floor{\sqrt{n}}$ and is said to be in domination type-II if $\gamma(\overline{G}) \le \floor{\sqrt{n}}$.
	\end{definition}
	
	Observe that every graph $G$ is in domination type-I or domination type-II or both.
	
	\begin{theorem}\label{thm:DomTypeDecision}
		Let $G$ be a graph of order $n$ and let $G \notin SCC(2)$. Then whether $G$ is in domination type-I or in domination type-II can be determined in polynomial time.
	\end{theorem}
	
	\begin{proof}
		If $G$ is a disconnected graph, then it follows from Theorem~\ref{thm:SuffiCondnDom} that $G$ is in domination type-II. Similarly, if $\overline{G}$ is a disconnected graph, then $G$ is in domination type-I.
		
		Suppose both $G$ and $\overline{G}$ are connected. It follows from Theorem~\ref{thm:SuffiCondnDom} that if $diam(G) \ge 3$, then $G$ is in domination type-II. Similarly, if $diam(\overline{G}) \ge 3$, then $G$ is in domination type-I.
		
		If $diam(G) = 1$ or $rad(G)=1$, then $\gamma(G) = 1$, and hence, $G$ is in domination type-I. Similarly, if $diam(\overline{G}) = 1$ or $rad(\overline{G})=1$, then $\gamma(\overline{G}) = 1$ and hence, $G$ is in domination type-II.
		
		In all other cases, $G \in SCC(2)$. Also, given a graph $G$, it can be decided in polynomial time whether $G$ is connected or not. If $G$ is connected, then $diam(G)$ and $rad(G)$ can be computed in polynomial time. Hence, the result follows.
	\end{proof}
	
	For graphs in $SCC(2)$, we propose the following conjecture.
	
	\begin{conjecture}
		If $G \in SCC(2)$, then $\gamma(G) \le \ceil{\sqrt{n}}$ and $\gamma(\overline{G}) \le \ceil{\sqrt{n}}$.
	\end{conjecture}
	
	Note that $C_5$ is an example of a graph in $SCC(2)$ which satisfies the conjecture.
	
	\section{Conclusion}
	Using the multiplicative version of the Nordhaus-Gaddum type result for the domination number $\gamma$ of a graph $G$, we have obtained a substantially improved upper bound for $\gamma$ and we have proved that for most of the graphs, this bound holds. An exhaustive survey of Nordhaus-Gaddum type results for various graph theoretic parameters is given in \cite{Aouchiche2013NordGaddSurvey}. This leads to two potential directions for further research, which are given below.
	
	\textbf{Problem 1:} Determine graph theoretic parameters for which an improved upper or lower bound can be obtained using Nordhaus-Gaddum type results.
	
	\textbf{Problem 2:} For each of the parameters identified in Problem 1, investigate the structural properties of graphs satisfying the improved bound.
	
	\bibliographystyle{SK}
	\bibliography{Domination}
		
\end{document}